\documentclass[12pt, reqno]{amsart}
\usepackage{amsmath, amsthm, amscd, amsfonts, amssymb, graphicx, color}
\usepackage[bookmarksnumbered, colorlinks, plainpages]{hyperref}
\input{mathrsfs.sty}
\hypersetup{colorlinks=true,linkcolor=red, anchorcolor=green, citecolor=cyan, urlcolor=red, filecolor=magenta, pdftoolbar=true}
\newtheorem{theorem}{Theorem}[section]
\newtheorem{lemma}[theorem]{Lemma}
\newtheorem{proposition}[theorem]{Proposition}
\newtheorem{corollary}[theorem]{Corollary}
\theoremstyle{definition}

\theoremstyle{remark}
\newtheorem{remark}[theorem]{Remark}
\begin{document}

\title[Asymmetric Choi--Davis inequalities]{Asymmetric Choi--Davis inequalities}
\author[M. Kian, M.S. Moslehian, R. Nakamoto]{M. Kian$^1$, M. S. Moslehian$^2$, \MakeLowercase{and} R. Nakamoto$^3$}

\address{$^1$Department of Mathematics, University of Bojnord, P. O. Box 1339, Bojnord 94531, Iran}
\email{kian@ub.ac.ir }

\address{$^2$Department of Pure Mathematics, Ferdowsi University of Mashhad,
Center of Excellence in Analysis on Algebraic Structures (CEAAS),
P.O. Box 1159, Mashhad 91775, Iran}
\email{moslehian@um.ac.ir, moslehian@yahoo.com}

\address{$^3$Daihara-cho, Hitachi, Ibaraki 316-0021, Japan}
\email{r-naka@net1.jway.ne.jp}

\subjclass[2010]{Primary 47A63; Secondary 15A60, 47B65.}

\keywords{Choi--Davis inequality, positive linear map, Kadison inequality, Kantorovich constant.}

\begin{abstract}
Let $\Phi$ be a unital positive linear map and let $A$ be a positive invertible operator. We prove that there exist partial isometries $U$ and $V$ such that
\[ |\Phi(f(A))\Phi(A)\Phi(g(A))|\leq U^*\Phi(f(A)Ag(A))U
\]
and
\[\left|\Phi\left(f(A)\right)^{-r}\Phi(A)^r\Phi\left(g(A)\right)^{-r}\right|\leq V^*\Phi\left(f(A)^{-r}A^rg(A)^{-r}\right)V\]
hold under some mild operator convex conditions and some positive numbers $r$. Further, we show that if $f^2$ is operator concave, then
$$ |\Phi(f(A))\Phi(A)|\leq \Phi(Af(A)).$$
In addition, we give some counterparts to the asymmetric Choi--Davis inequality and asymmetric Kadison inequality. Our results extend some inequalities due to Bourin--Ricard and Furuta.
\end{abstract}
\maketitle

%---------------------------------------------------------------------------------------%
%%%%%%%%%%%%%%%%%%%%%%%%%%%%%%%%%
%%%%%%%%%%%%%%%%%%%%%%%%%%%%%%%%%
\section{introduction}

Throughout the paper, let $\mathbb{B}(\mathscr{H})$ stand for the $C^*$-algebra of all bounded linear operators on a Hilbert space $\mathscr{H}$ with the identity $I$. It is identified by the full matrix algebra $\mathbb{M}_n$ when $\dim \mathscr{H}=n$. A capital letter displays an operator in $\mathbb{B}(\mathscr{H})$. The usual (L\"owner) order on the real space of all self-adjoint operators is denoted by $\leq$; in particular, we write $A\geq 0$ when $A$ is a positive operator (positive semidefinite matrix). When $mI\leq A\leq MI$, we write $m\leq A\leq M$ for simplicity. A map $\Phi$ defined on $\mathbb{B}(\mathscr{H})$ is called \emph{positive} whenever it takes positive operators to positive operators.

In the classical probability theory, the variance of a random variable $X$ is defined by $\mathrm{Var}(X)=\mathbf{E}(X^2)-\mathbf{E}(X)^2$, where $\mathbf{E}$ is the expectation value. One of the basic properties of this quantity is its positivity. As a noncommutative extension, the operator valued map
$\mathrm{Var}(A)=\Phi(A^2)-\Phi(A)^2$
is said to be the \emph{variance} of the self-adjoint operator $A$, where $\Phi$ is a unital positive linear map. The celebrated Kadison inequality asserts that $\mathrm{Var}(A)$ is a positive operator, that is, \[\Phi(A)^2\leq \Phi(A)^2.\]

A continuous real function $f$ defined on an interval $J \subseteq \mathbb{R}$ is called \emph{operator convex} if $f(\lambda A+(1-\lambda)B)\leq \lambda f(A)+(1-\lambda)f(B)$ for all self-adjoint operators $A$ and $B$ with spectra in $J$ and all $\lambda\in[0,1]$. It is called \emph{operator concave} whenever $-f$ is operator convex. It can be shown that a continuous function $f$ defined on an interval $J$ is operator convex if and only if the so-called \emph{Choi--Davis inequality}
\begin{eqnarray}\label{chda}
f(\Phi(A))\leq \Phi(f(A))
\end{eqnarray}
holds for all self-adjoint operators $A$ with spectrum in $J$ and for all unital positive linear maps $\Phi$. In fact, Davis \cite{DAV} proved that \eqref{chda} holds when $f$ is an operator convex function and $\Phi$ is a completely positive linear map. Choi \cite{CHO} showed that inequality \eqref{chda} remains true for all positive unital linear maps $\Phi$ and all operator convex functions $f$.

If $f$ is convex but not operator convex, it is shown in \cite{BS} that the Choi--Davis inequality remains valid for every $2\times 2$ Hermitian matrix $A$. Bourin and Lee in the nice survey \cite{Bou} gave a variety of Choi--Davis type inequalities for general convex or concave functions. Niezgoda \cite{Niz} utilized generalized inverses of some linear operators and presented a refinement of the Choi--Davis inequality.
The following inequalities are special cases of the Choi--Davis inequality:
 \begin{align}\label{c2}
 \Phi(A)^p\leq\Phi(A^p)\,\, (1\leq p\leq 2\,\, \mbox{or}\,\, -1\leq p\leq0) \quad {\rm and}\quad
 \Phi(A)^p\geq\Phi(A^p)\,\, (0\leq p\leq1).
 \end{align}

Sharma et al. \cite{SDK} gave a generalization of the Kadison inequality by showing the positivity of the operator matrix
\[\left[\begin{matrix}
 I & \Phi(A) & \cdots & \Phi\left(A^{r}\right) \\
 \Phi(A) & \Phi(A^2) & \cdots & \Phi\left(A^{r+1}\right) \\
 \vdots & \vdots & \ddots & \vdots \\
 \Phi\left(A^{r}\right) & \Phi\left(A^{r+1}\right) & \cdots & \Phi\left(A^{2r}\right)
 \end{matrix}\right].\]

Bourin and Ricard \cite{BR} utilized the celebrated Furuta inequality and presented an asymmetric Kadison inequality by showing that if $\gamma \in [0,1]$, then
\begin{align}\label{asy}
\left|\Phi(X^\gamma)\Phi(X)\right|\leq \Phi(X)^{1+\gamma}
\end{align}
holds for every positive operator $X$. This further implies a noncommutative version of Chebychev's inequality as follows:
\begin{align}\label{asy2}
\left|\Phi(X^\alpha)\Phi(X^\beta)\right|\leq \Phi(X^{\alpha+\beta})
\end{align}
for all $0\leq \alpha\leq \beta$. Sharma and Thakur \cite{ST} proved that a unital positive linear map $\Phi$ on $\mathbb{M}_2$ preserves the commutativity of operators and used this fact to establish some results analogue to \eqref{asy2} with $\Phi(X^\beta)\Phi(X^\alpha)$ instead of $\left|\Phi(X^\alpha)\Phi(X^\beta)\right|$.

An extension of \eqref{asy2} was presented by Furuta \cite{Fur} as follows:
\begin{align}\label{asy222}
 \left|\Phi(X^\alpha)^\gamma\Phi(X^\beta)^\gamma\right|\leq \Phi(X^{(\alpha+\beta)\gamma}),
\end{align}
when $0\leq \alpha\leq \beta$ and $\frac{\beta}{\alpha+\beta}\leq \gamma\leq \frac{2\beta}{\alpha+\beta}$. In fact, he gave a result interpolating \eqref{asy2} and the first inequality in \eqref{c2}. Furthermore, he showed that under the same conditions as above, the inequality
\begin{align}\label{asy33}
 \left|\Phi(X^{-\alpha})^{-\gamma}\Phi(X^\beta)^\gamma\right|\leq \Phi(X^{(\alpha+\beta)\gamma})
\end{align}
is true.

Some further extensions of \eqref{asy222} have been discussed in \cite{YG}.

The organization of the paper is as follows. In Section 2, we examine possible extensions of the classical Chebyshev inequality and then present some asymmetric Choi--Davis inequalities, which extend inequalities \eqref{asy2} and \eqref{asy33} in some certain directions. More precisely, we prove that if $\Phi$ is a unital positive linear map and $A$ is a positive invertible operator, then under some mild operator convex conditions and some positive numbers $r$, there exist partial isometries $U$ and $V$ such that
\[ |\Phi(f(A))\Phi(A)\Phi(g(A))|\leq U^*\Phi(f(A)Ag(A))U
\]
and
\[\left|\Phi\left(f(A)\right)^{-r}\Phi(A)^r\Phi\left(g(A)\right)^{-r}\right|\leq V^*\Phi\left(f(A)^{-r}A^rg(A)^{-r}\right)V.\]
In Section 3, we give some counterparts to the asymmetric Choi--Davis inequality and asymmetric Kadison inequality \eqref{asy2}. Among other things, we show that, for
every positive invertible operator $A$ and certain real numbers $\alpha, \beta$, and $\gamma$, there exists a partial isometry $W$ such that
 \begin{align}\label{me1}
\Phi(A^{\alpha+\beta+\gamma})\leq K\,W \left|\Phi(A^\alpha)\Phi(A^\beta)\Phi(A^\gamma)\right|W^*
 \end{align}
for some Kantorovich type constant $K$.
\section{Asymmetric Choi--Davis inequality}

Assume that $\{a_i\}$ and $\{b_i\}$\ $(i=1,\ldots,k)$ are increasing sequences of positive real numbers. The classical Chebyshev inequality asserts that
\[\left(\frac{1}{k}\sum_{i=1}^{k}a_i\right)\left(\frac{1}{k}\sum_{i=1}^{k}b_i\right)\leq \frac{1}{k}\sum_{i=1}^{k}a_ib_i.\]
If one of the sequences is decreasing, then the reverse inequality holds.

Assume that $\Phi$ is a unital positive linear map. The operator extension
\begin{align}\label{ch-op1}
|\Phi(B)\Phi(A)|\leq \Phi(A^{1/2}BA^{1/2})
\end{align}
 does not hold in general. To see this, assume that the unital positive linear map $\Phi:\mathbb{M}_3\to\mathbb{M}_2$ is defined by $\Phi([a_{ij}]_{1\leq i,j\leq 3})= [a_{ij}]_{1\leq i,j\leq 2}$ and consider the positive matrices
\[B=\left[\begin{array}{ccc}
 2 & 1 & 1 \\ 1 & 2 & 0 \\ 1 & 0 & 1
\end{array}\right] \quad\mbox{and}\quad A=\left[\begin{array}{ccc}
 2 &0 & 0 \\ 0 & 2 & 1 \\ 0 & 1 & 3
\end{array}\right].\]
Then $|\Phi(B)\Phi(A)|=\left[\begin{array}{cc}
 4 & 2 \\ 2 & 4
\end{array}\right] \nleqslant \left[\begin{array}{cc}
 4 & 2.4 \\ 2.4 & 3.89
\end{array}\right]=\Phi(A^{1/2}BA^{1/2})$.

Another possible extension is
\begin{align}\label{ch-op2}
 \Phi(A)\Phi(B)\Phi(A)\leq \Phi(ABA).
\end{align}
 This is not true in general, too. Using the same unital positive linear map $\Phi$ and positive matrices $A$ and $B$ as above, we get
\[\Phi(A)\Phi(B)\Phi(A)=\left[\begin{array}{cc}
 8 & 4 \\ 4 & 8
\end{array}\right]\nleqslant \left[\begin{array}{cc}
 8& 6 \\ 6 & 9
\end{array}\right]= \Phi(ABA).\]

Bourin and Ricard \cite{BR} showed that in the case when $B:=A^\gamma$ and $\gamma\in[0,1]$, inequality \eqref{ch-op1} is valid. They also presented a variant of \eqref{ch-op2} in the setting of complex matrices $\mathbb{M}_n$ as
$$\Phi(A)\Phi(B)\Phi(A)\leq V\Phi(ABA)V^*$$
for some unitary matrix $V$, where $(A,B)$ is a pair of matrices with the property that $A=h_1(C)$ and $B=h_2(C)$ for some nonnegative, nondecreasing, and continuous functions $h_1$ and $h_2$.

We note that a weaker version of \eqref{ch-op2} as
$$\Phi(A)\Phi(B)^{-1}\Phi(A)\leq \Phi(AB^{-1}A)$$
is valid in general for all positive operators $A$ and $B$. This is, a special case of the inequality
$g(\Phi(A),\Phi(B))\leq \Phi(g(A,B))$, which holds for every operator perspective function $g$ defined by
$g(A,B)=A^{1/2}f(A^{-1/2}BA^{-1/2})A^{1/2}$, where $f:[0,\infty)\to[0,\infty)$ is an operator convex function; see \cite{MK}.

We need some known properties of operator concave functions. The next lemma can be found in \cite{FMPS}; see \cite[Theorem 1.13 and Corollary 1.14]{FMPS}.

\begin{lemma}\label{lfmps}
 Suppose that $f:(0,\infty)\to (0,\infty)$ is a continuous function. The followings assertions are equivalent:\\
 \rm{(i)} \ $f(t)$ is operator concave;\\
 \rm{(ii)} \ $f(t)$ is operator monotone;\\
 \rm{(iii)} \ $\frac{t}{f(t)}$ is operator monotone;\\
 \rm{(iv)} \ $tf(t)$ is operator convex.
\end{lemma}
%%%%%%%%%%%%%%%%%%%%%%%%%%%%%%%%%%%%%%%%%%%%%%%%%%%%%%%%%%%%%%%
Our first main result presents a variant of \eqref{ch-op2}.

\begin{theorem}\label{th-nnn1}
Assume that $\Phi$ is a unital positive linear map and that $0\leq r\leq 1/2$. If $f,g:(0,\infty)\to (0,\infty)$ are operator convex functions, then, for every positive invertible operator $A$, there exists a partial isometry $V$ such that
\begin{align}\label{po1}
\left|\Phi\left(f(A)\right)^{-r}\Phi(A)^r\Phi\left(g(A)\right)^{-r}\right|\leq V^*\Phi\left(f(A)^{-r}A^rg(A)^{-r}\right)V
\end{align}
provided that $\frac{f(t)g(t)}{t}$ is operator concave. If $f,\ g,$ and $\frac{t}{f(t)g(t)}$ are operator concave, then the reverse inequality holds.
\end{theorem}

 As a special case of Theorem \ref{th-nnn1}, assume that $f$ is operator convex and put $g(t)=1$.
 In this case, taking account of Lemma \ref{lfmps}, the operator concavity of $\frac{f(t)g(t)}{t}$ is automatically satisfied. Note that the following equivalence assertions are derived from Lemma \ref{lfmps} and the operator concavity of $t\mapsto t^{r}$:
 $$\mbox{$f$ is operator convex}\,\, \Longleftrightarrow\,\,\mbox{$\frac{f(t)}{t}$ is operator concave}\,\, \Longrightarrow\,\, \mbox{$\frac{f(t)^r}{t^r}$ is operator concave}.$$
 The last one further implies the operator convexity of the function $\frac{t^r}{f(t)^r}$.
Hence, we obtain the next result.
\begin{corollary}
Let $\Phi$ be a unital positive linear map, let $0\leq r\leq 1/2$, and let $f:(0,\infty)\to(0,\infty)$ be an operator convex function. Then
 $$\left|\Phi\left(f(A)\right)^{-r}\Phi(A)^r\right|\leq \Phi\left(A^rf(A)^{-r}\right)$$
for every positive invertible operator $A$.
\end{corollary}

If $\gamma\in[0,1]$, then with $f(t)=t^{-\gamma}$ and $g(t)=1$, Theorem \ref{th-nnn1} concludes a variant of \cite[Theorem 2.1]{Fur}.
\begin{corollary}
If $\gamma\in[0,1]$ and $0\leq r\leq 1/2$, then
 $$\left|\Phi\left(A^{-\gamma}\right)^{-r}\Phi(A)^r\right|\leq \Phi\left(A^{(1+\gamma)r}\right)$$
for every unital positive linear map $\Phi$ and every positive invertible operator $A$.
\end{corollary}

\textbf{Proof of Theorem \ref{th-nnn1}.} Since $f$ and $g$ are operator convex, the Choi--Davis inequality together with the operator monotonicity of $t\mapsto t^{2r}$ imply that
\begin{align}\label{qwe1}
\Phi(f(A))^{-2r}\leq f(\Phi(A))^{-2r}\quad\mbox{and}\quad\Phi(g(A))^{-2r}\leq g(\Phi(A))^{-2r}.
\end{align}
Therefore
 \begin{align}\label{lkll1}
 |\Phi(f(A))^{-r}\Phi(A)^r\Phi(g(A))^{-r}|&=\left\{\Phi(g(A))^{-r}\Phi(A)^r
 \Phi(f(A))^{-2r}\Phi(A)^r\Phi(g(A))^{-r}\right\}^{1/2}\nonumber\\
 &\leq
 \left\{\Phi(g(A))^{-r}\Phi(A)^{r}f(\Phi(A))^{-2r}\Phi(A)^{r}\Phi(g(A))^{-r}\right\}^{1/2}\nonumber\\
 &=\left|f(\Phi(A))^{-r}\Phi(A)^{r}\Phi(g(A))^{-r}\right|.
 \end{align}
There exists a partial isometry $V$ such that
\begin{align}\label{lkll2}
 \left|f(\Phi(A))^{-r}\Phi(A)^{r}\Phi(g(A))^{-r}\right|=V^* \left|\Phi(g(A))^{-r}\Phi(A)^{r}f(\Phi(A))^{-r}\right| V.
\end{align}
 Moreover, by employing \eqref{qwe1}, we get
 \begin{align}\label{lkll3}
 \left|\Phi(g(A))^{-r}\Phi(A)^{r}f(\Phi(A))^{-r}\right|&=
 \left\{f(\Phi(A))^{-r}\Phi(A)^{r}\Phi(g(A))^{-2r}\Phi(A)^{r}f(\Phi(A))^{-r}\right\}^{1/2}\nonumber\\
 &\leq \left\{f(\Phi(A))^{-r}\Phi(A)^{r}g(\Phi(A))^{-2r}\Phi(A)^{r}f(\Phi(A))^{-r}\right\}^{1/2}\nonumber\\
 &=\left(f(\Phi(A))^{-1}\Phi(A)g(\Phi(A))^{-1}\right)^{r}.
 \end{align}
 Since $t\mapsto t^r$ is operator monotone, it follows from Lemma \ref{lfmps} that it is operator concave. Hence, if the function $t\mapsto \frac{f(t)g(t)}{t}$ is operator concave, then so is $t\mapsto \frac{f(t)^rg(t)^r}{t^r}$. This guarantees the operator convexity of the function $t\mapsto \frac{t^r}{f(t)^rg(t)^r}$. Hence, the Choi--Davis inequality yields that
 \begin{align}\label{lkll4}
 f(\Phi(A))^{-r}\Phi(A)^rg(\Phi(A))^{-r}\leq \Phi\left(f(A)^{-r}A^rg(A)^{-r}\right).
 \end{align}
Inequality \eqref{po1} is deduced by combining \eqref{lkll1}, \eqref{lkll2}, \eqref{lkll3}, and \eqref{lkll4} together.

If $f$ and $g$ are operator concave, then
\begin{align*}
\Phi(f(A))^{-2r}\geq f(\Phi(A))^{-2r}\quad\mbox{and}\quad\Phi(g(A))^{-2r}\geq g(\Phi(A))^{-2r}.
\end{align*}
Furthermore, if $\frac{t}{f(t)g(t)}$ is operator concave, then the function $\frac{t^r}{f(t)^rg(t)^r}$ is also operator concave. A similar argument as in the proof of \eqref{po1} shows that the reverse inequality of \eqref{po1} holds.

%%%%%%%%%%%%%%%%%%%%%%%%%%%%%%%%%%%%%%%%%%%%%%%%%%%%%%%%%%%%%%%%%%%%%%%%%%%%%%%%%%%%%%%%%%%%%%
The following theorem gives a Choi--Davis type asymmetric inequality. Further, it provides a generalization of the asymmetric Kadison inequality \eqref{asy2}.

\begin{theorem}\label{th-nn1}
 Assume that $f, g:(0,\infty)\to (0,\infty)$ are continuous functions and that $\Phi$ is a unital positive linear map. If $f^2$ and $g^2$ are operator concave, then, for every positive invertible operator $A$, there exist partial isometries $U$ and $V$ such that
 \begin{align}\label{tt1m1}
 |\Phi(f(A))\Phi(A)\Phi(g(A))|\leq U^*\Phi(f(A)Ag(A))U
 \end{align}
holds provided that $tf(t)g(t)$ is operator convex, and
 \begin{align}\label{tt1m2}
\left|\Phi(f(A))^{-1}\Phi(A)\Phi(g(A))^{-1}\right|\geq V^*\Phi(f(A)^{-1}Ag(A)^{-1})V
 \end{align}
holds provided that $t/f(t)g(t)$ is operator concave.
\end{theorem}
\begin{proof}
The proof is similar to that of Theorem \ref{th-nnn1} and so we omit its details. Just note that the Kadison inequality and the operator concavity of $f^2$ and $g^2$ imply that
 \begin{align*}
 \Phi(f(A))^2\leq\Phi(f(A)^2)\leq f^2(\Phi(A))\quad\mbox{and}\quad \Phi(g(A))^2\leq\Phi(g(A)^2)\leq g^2(\Phi(A))
 \end{align*}
 and
 $$\Phi(f(A))^{-2}\geq\Phi(f^2(A))^{-1}\geq f(\Phi(A))^{-2}\,\,\mbox{and}\,\, \Phi(g(A))^{-2}\geq\Phi(g^2(A))^{-1}\geq g(\Phi(A))^{-2}.$$
\end{proof}
%%%%%%%%%%%%%%%%%%%%%%%%%%%%%%%%%%%%%%%%%%%%%%%%%%%%%%%%%%%%%%%%%%%%%

If $f^2(t)$ is operator concave, then $f(t)$ is also operator concave. Lemma \ref{lfmps} concludes that the functions $h_1(t)=t f(t)$ and $h_2(t)=\frac{t}{f(t)}$ are operator convex and operator concave, respectively. Therefore, if $g(t)=1$, then the conditions of Theorem \ref{th-nn1} are fulfilled automatically. Thus we arrive at the following asymmetric Choi--Davis inequality.

%%%%%%%%%%%%%%%%%%%%%%%%%%%%%%%%%%%%%%%%%%%%%%%%%%%%%%%%%%%%%%%%%%%%%%%%%%%%%%%%
\begin{corollary}\label{th-n1}
 Assume that $f:(0,\infty)\to (0,\infty)$ is a continuous function. If $f^2$ is operator concave, then
 \begin{align}\label{t1m1}
 |\Phi(f(A))\Phi(A)|\leq \Phi(Af(A))
 \end{align}
 and
 \begin{align}\label{t1m2}
|\Phi(f(A))^{-1}\Phi(A)|\geq \Phi\left(A f(A)^{-1}\right)
 \end{align}
 for every unital positive linear map $\Phi$ and positive invertible operator $A$. In addition, if $f(t)\geq1$, then
 $$\Phi(A f(A)^{-1})\leq |\Phi(f(A))^{-1}\Phi(A)|\leq |\Phi(f(A)) \Phi(A)|\leq \Phi(A f(A)).$$
\end{corollary}

Let $f(t)=t^\gamma$, where $0\leq \gamma\leq 1/2$. Then $f^2$ is operator concave. Hence, as a consequence of Corollary \ref{th-n1}, we obtain inequality \eqref{asy}.
\begin{corollary} Inequalities
\begin{align*}
 |\Phi(A^\gamma)\Phi(A)|\leq \Phi(A^{\gamma+1})
\end{align*}
and
\begin{align*}
|\Phi(A^\gamma)^{-1}\Phi(A)|\geq \Phi(A^{1-\gamma})
\end{align*}
hold for each unital positive linear map $\Phi$, each positive invertible operator $A$, and each $0\leq \gamma\leq 1/2$.
\end{corollary}

%%%%%%%%%%%%%%%%%%%%%%%%%%%%%%%%%%%%%%%%%%%%%%%%%%%%%%%%%%%%%%%%%%%%
Our next result reads as follows.

\begin{proposition}
Let $\Phi$ be a unital positive linear map, let $f:(0,\infty)\to (0,\infty)$ be a continuous function, and let	 $0\leq r\leq 1/2$. If $f^2$ is operator concave, then
 $$\left|\Phi\left(f(A)^{-1}\right)^{-r}\Phi(A)^r\right|\leq \Phi(Af(A))^r$$
for every positive invertible operator $A$.
\end{proposition}
\begin{proof}
 Note that the functions $t\mapsto t^{r}$ and $t\mapsto t^{-2r}$ are operator concave and operator convex, respectively. Therefore, the Choi--Davis inequality implies that
 $$\Phi\left(f(A)^{-1}\right)^{-2r}\leq \Phi\left(f(A)^{2r}\right)\leq \Phi\left(f(A)^2\right)^r\leq f(\Phi(A))^{2r},$$
 where the last inequality follows from the operator concavity of $f^2$. Hence
 \begin{align*}
 \left|\Phi\left(f(A)^{-1}\right)^{-r}\Phi(A)^r\right|&=
 \left\{\Phi(A)^r\Phi\left(f(A)^{-1}\right)^{-2r}\Phi(A)^r\right\}^{1/2}\\
 &\leq \left\{\Phi(A)^r f(\Phi(A))^{2r}\Phi(A)^r\right\}^{1/2}\\
 &=\Phi(A)^rf(\Phi(A))^{r}\\
 &=h(\Phi(A))^r,
 \end{align*}
 where $h(t)=tf(t)$ is operator convex by Lemma \ref{lfmps}. The Choi--Davis inequality and the operator monotonicity of $t\mapsto t^{r}$ give
 $$h(\Phi(A))^r\leq \Phi(h(A))^r=\Phi(Af(A))^r.$$
 \end{proof}
%%%%%%%%%%%%%%%%%%%%%%%%%%%%%%%%%%%%%%%%%%%%%%%%%%%%%%%%%%%%%%%%%%%%%%%%%

\section{Counterparts to the Choi--Davis inequality}

In this section, we present counterparts to some Choi--Davis inequalities. In particular, we give a converse to \eqref{t1m1} in the next theorem. First, we recall a result from \cite{MPS}.

\begin{lemma}\cite[Corollary 2.4]{MPS}\label{lm-rev-j}
Let $\Phi$ be a unital positive linear map, let $f:[m,M]\to(0,\infty)$ be a continuous
function with $0<m<M$, and let $A$ be a positive operator with $m\leq A\leq M$. If $f$ is strictly concave, then
$$\Phi(f(A))\geq K_1(m,M,f) f(\Phi(A)),$$
where
$$K_1(m,M,f)=\min_{t\in[m,M]}\left\{\frac{(M-t)f(m)+(t-m)f(M)}{(M-m)f(t)}\right\}.$$
If $f$ is strictly convex, then
$$\Phi(f(A))\leq K_2(m,M,f) f(\Phi(A)),$$
where
$$K_2(m,M,f)=\max_{t\in[m,M]}\left\{\frac{(M-t)f(m)+(t-m)f(M)}{(M-m)f(t)}\right\}.$$
\end{lemma}
The special case when $f$ is the power function reads as follows. Recall that the generalized Kantorovich constant $\kappa(h, p)$ is defined by
\[\kappa(h,p):={h^p-1\over{(p-1)(h-1)}}\left({p-1\over p}{h^p-1\over{h^p-h}}\right)^p.\]

\begin{lemma}\label{lm-rev-choi}\cite[Lemma 4.3]{FMPS}
Let $\Phi$ be a unital positive linear map and let $A$ be a positive invertible operator with $0<m\leq A\leq M$. Then,
\begin{itemize}
\item[(i)] if $p>1$, then $\Phi(A^p)\leq K(m,M, p)\Phi(A)^p$;
\par
\item[(ii)] if $0<p<1$, then $K(m,M, p)\Phi(A)^p\leq \Phi(A^p)$,
where
\[K(m,M, p)=\kappa(M/m,p)={mM^p-Mm^p\over{(p-1)(M-m)}}\left({p-1\over p}{M^p-m^p\over{mM^p-Mm^p}}\right)^p.\]
\end{itemize}
\end{lemma}

The next theorem presents a reverse of inequality \eqref{t1m1}.
\begin{theorem}\label{thhh1}
Let $\Phi$ be a unital positive linear map, let $0<m<M$, and let $f:[m,M]\to(0,\infty)$ be a continuous function such that $f^2$ is strictly concave. If $A$ is a positive operator with $m\leq A\leq M$, then
$$\Phi(Af(A))\leq K |\Phi(f(A))\Phi(A)|,$$
where $$K=\kappa(f(M)/f(m),2)^{1/2} K_1(m,M,f^2)^{-1/2} K_2(m,M,tf(t)).$$
\end{theorem}

\begin{proof}
Assume that $f^2$ is a concave continuous function. Using Lemma \ref{lm-rev-j}, we obtain $\Phi(f(A)^2)\geq K_1(m,M,f^2) f^2(\Phi(A))$. In addition, Lemma \ref{lm-rev-choi} gives
$$\Phi(f(A)^2)\leq \kappa(f(M)/f(m),2) \Phi(f(A))^2.$$
 Hence
\begin{align}\label{e1e}
 f^2(\Phi(A))\leq \kappa(f(M)/f(m),2) K_1(m,M,f^2)^{-1} \Phi(f(A))^2.
\end{align}
Therefore
\begin{align}\label{ee1}
 \Phi(A)f(\Phi(A))&=\left\{\Phi(A)f(\Phi(A))^2\Phi(A)\right\}^{1/2}\nonumber\\
 &\leq \kappa(f(M)/f(m),2)^{1/2} K_1(m,M,f^2)^{-1/2} \left\{\Phi(A)\Phi(f(A))^2\Phi(A)\right\}^{1/2}\nonumber\\
 &=\kappa(f(M)/f(m),2)^{1/2} K_1(m,M,f^2)^{-1/2} |\Phi(f(A))\Phi(A)|.
\end{align}
Moreover, if the continuous function $tf(t)$ is convex, then an application of Lemma \ref{lm-rev-j} yields that
\begin{align}\label{ee2}
\Phi(A)f(\Phi(A))\geq K_2(m,M,tf(t))^{-1} \Phi(Af(A)).
\end{align}
Then \eqref{ee1} and \eqref{ee2} give the desired result.
\end{proof}

%%%%%%%%%%%%%%%%%%%%%%%%%%%%%%%%%%%%%%%%%%%%%%%%%%%%%%%%%%%%%%%%%%%
In the special case when $f$ is the power function, Theorem \ref{thhh1} turns to the following corollary. It provides a counterpart to the asymmetric Kadison inequality \cite[Theorem 1.1]{BR}.
\begin{corollary}\label{nakamoto}
Let $\Phi$ be a unital positive linear map and let $\gamma\in[0,1]$. If $A$ is a positive operator with $0<m\leq A\leq M$, then
\[\Phi\left(A^{1+\gamma}\right)\leq \kappa(h,1+\gamma) \kappa(h^\gamma,2)^{1\over2}\kappa(h^2,\gamma)^{-1\over2}|\Phi(A^\gamma)\Phi(A)|.\]
 In particular, for every $0\leq \alpha\leq \beta$, it follows that
\begin{align}\label{m4}
\Phi(A^{\alpha+\beta})\leq \kappa(h^\beta,1+{\alpha/\beta})\kappa(h^\beta, {2\alpha/ \beta})^{-1\over2}\kappa(h^\alpha, 2)^{1\over2}|\Phi(A^\alpha)\Phi(A^\beta)|.
\end{align}
\end{corollary}
%%%%%%%%%%%%%%%%%%%%%%%%%%%%%%%%%%%%%%%%%%%%%%%%%%%%%%%%%%%%%%%%%%%%%%%%

A version of \eqref{m4} including three parameters, gives a counterpart to \cite[Proposition 1.3]{BR}.
%%%%%%%%%%%%%%%%%%%%%%%%%%%%%%%%%%%%%%%%%%%%%%%%%%%%%%%%%%%%%%%%%%%%%%%%%%%%%%%%%%
\begin{theorem}\label{k1}
 Let $\Phi$ be a unital positive linear map. If $\alpha,\beta,\gamma\geq0$ with $\min\{\alpha,\beta\}\leq \frac{\gamma}{2}$ and $\max\{\alpha,\beta\}\leq \gamma$, then, for every positive invertible operator $A$ with $0<m\leq A\leq M$, there exists a partial isometry $U$ such that
 \begin{align}\label{me1}
\Phi(A^{\alpha+\beta+\gamma})\leq K\,U \ \left|\Phi(A^\alpha)\Phi(A^\beta)\Phi(A^\gamma)\right|U^*,
 \end{align}
where
{\small \begin{align*}
 K=\kappa(h^\alpha,2)^\frac{1}{2}\kappa(h^\beta,2)^\frac{1}{2}
\kappa(h^\gamma,2\beta/\gamma)^\frac{-1}{2}\kappa(h^\gamma,2 \alpha/\gamma)^\frac{-1}{2} \kappa(h^\gamma,1+(\alpha+\beta)/\gamma).
 \end{align*}}
\end{theorem}
\begin{proof}
 Without loss of generality, we may assume that $ \beta\leq \alpha$. First, we assume that $\gamma=1$. We then have from our hypotheses that $\beta\leq 1/2$ and $\beta\leq \alpha\leq 1$. By virtue of $2\beta\leq 1$ and $m\leq A\leq M$, Lemma \ref{lm-rev-choi} ensures that $\Phi(A)^{2\beta}\leq \kappa(h,2\beta)^{-1}\Phi(A^{2\beta})$. Using Lemma \ref{lm-rev-choi} once more, we obtain
 $\Phi(A^{2\beta})\leq \kappa(h^\beta,2)\Phi(A^\beta)^2$. Therefore, we get
 \begin{align}\label{q1q}
 \Phi(A)^{2\beta}\leq \kappa(h,2\beta)^{-1}\kappa(h^\beta,2)\Phi(A^\beta)^2.
 \end{align}
 Utilizing the operator monotonicity of $t\mapsto t^\frac{1}{2}$ and \eqref{q1q}, we can write
 \begin{align}\label{n10}
 \left|\Phi(A)^{1+\beta}\Phi(A^\alpha)\right|
 & =\left\{\Phi(A^\alpha)\Phi(A)^{2+2\beta}\Phi(A^\alpha)\right\}^\frac{1}{2}\nonumber\\
 & =\left\{\Phi(A^\alpha)\Phi(A)\Phi(A)^{2\beta}\Phi(A)\Phi(A^\alpha)\right\}^\frac{1}{2}\nonumber\\
 &\leq \kappa(h,2\beta)^\frac{-1}{2}\kappa(h^\beta,2)^\frac{1}{2}
 \left\{\Phi(A^\alpha)\Phi(A)\Phi(A^\beta)^2\Phi(A)\Phi(A^\alpha)\right\}^\frac{1}{2}\nonumber\\
 &=\kappa(h,2\beta)^\frac{-1}{2}\kappa(h^\beta,2)^\frac{1}{2}\left|\Phi(A^\beta)\Phi(A)\Phi(A^\alpha)\right|.
 \end{align}

From $\alpha\leq 1$ and Lemma \ref{lm-rev-choi}, we conclude that
$\Phi(A^\alpha)^2\geq \kappa(h^\alpha,2)^{-1}\Phi(A^{2\alpha})$ and $\Phi(A^{2\alpha})\geq \kappa(h^2,\alpha)\Phi(A^{2})^\alpha$. Therefore,
 \begin{align*}
 \left\{\Phi(A)^{1+\beta}\Phi(A^\alpha)^2\Phi(A)^{1+\beta}\right\}^\frac{1}{2}
 &\geq \kappa(h^\alpha,2)^\frac{-1}{2}\left\{\Phi(A)^{1+\beta}\Phi(A^{2\alpha})\Phi(A)^{1+\beta}\right\}^\frac{1}{2}\\
 &\geq \kappa(h^\alpha,2)^\frac{-1}{2}\kappa(h^2,\alpha)^\frac{1}{2}
 \left\{\Phi(A)^{1+\beta}\Phi(A^{2})^\alpha\Phi(A)^{1+\beta}\right\}^\frac{1}{2}\\
 &\geq
 \kappa(h^\alpha,2)^\frac{-1}{2}\kappa(h^2,\alpha)^\frac{1}{2}
 \left\{\Phi(A)^{1+\beta}\Phi(A)^{2\alpha}\Phi(A)^{1+\beta}\right\}^\frac{1}{2}\\
 &= \kappa(h^\alpha,2)^\frac{-1}{2}\kappa(h^2,\alpha)^\frac{1}{2}\Phi(A)^{1+\alpha+\beta}.
 \end{align*}
 The last inequality follows from the Kadison inequality and the operator monotonicity of $t\mapsto t^\alpha$. This implies that there exists a partial isometry $U$ such that
 \begin{align}\label{e1}
 \left|\Phi(A)^{1+\beta}\Phi(A^\alpha)\right|=U \left|\Phi(A^\alpha)\Phi(A)^{1+\beta}\right| U^*\geq \kappa(h^\alpha,2)^\frac{-1}{2} \kappa(h^2,\alpha)^\frac{1}{2} U \Phi(A)^{1+\alpha+\beta} U^*.
 \end{align}
 Moreover, since $1+\alpha+\beta>1$, Lemma \ref{lm-rev-choi} gives $\Phi(A)^{1+\alpha+\beta}\geq \kappa(h,1+\alpha+\beta)^{-1}\Phi(A^{1+\alpha+\beta})$. Therefore, from \eqref{e1}, we infer that
 \begin{align}\label{e2}
 \left|\Phi(A)^{1+\beta}\Phi(A^\alpha)\right|\geq \kappa(h^\alpha,2)^\frac{-1}{2} \kappa(h^2,\alpha)^\frac{1}{2} \kappa(h,1+\alpha+\beta)^{-1} U\Phi(A^{1+\alpha+\beta}) U^*.
 \end{align}
 Combining \eqref{e2} with \eqref{n10}, we deduce that
 \begin{align}\label{n133}
\Phi(A^{1+\alpha+\beta}) \leq K'\, U^* \left|\Phi(A^\beta)\Phi(A)\Phi(A^\alpha)\right| U,
\end{align}
 where
 $$K'=\kappa(h^\alpha,2)^\frac{1}{2}\kappa(h^\beta,2)^\frac{1}{2}
\kappa(h,2\beta)^\frac{-1}{2}\kappa(h,2 \alpha)^\frac{-1}{2} \kappa(h,1+\alpha+\beta).$$
 This proves the desired inequality \eqref{me1} in the case when $\gamma=1$.
 If $\gamma\neq 1$, then replace $\alpha$ and $\beta$ by $\alpha/\gamma$ and $\beta/\gamma$, respectively, and put $A^\gamma$ instead of $A$ in \eqref{n133} to get the result.
\end{proof}

%%%%%%%%%%%%%%%%%%%%%%%%%%%%%%%%%%%%%%%%%%%%%%%%%%%%%%%%%%%%%%%%%%%%%%%%%%%%%%
 \begin{remark}
Theorem \eqref{k1} with $\beta=0$ gives \eqref{m4} of Corollary \ref{nakamoto}.
 \end{remark}

%%%%%%%%%%%%%%%%%%%%%%%%%%%%%%%%%%%%%%%%%%%%%%%%%%%%%%%%%%%%%%%%%%%%%%%%%%%%%%%%%%%%%%%%%%%%%%%
%%%%%%%%%%%%%%%%%%%%%%%%%%%%%%%%%%%%%%%%%%%%%%%%%%%%%%%%%%%%%%%%%%%%%%%%%%%%%%%%%%%%%%%%%%%%
Next, we use a refinement of the Furuta inequality (see \cite{FNT,FNa}) to give a sharper inequality than \eqref{asy2}. To this end, we need some lemmas.

\begin{lemma}\label{MSM}
If $\Phi$ is a unital positive linear map, $A$ is a positive operator, and $1/2\leq r < 1$, then
\begin{align}\label{MSM1}
\Phi(A)^r-\Phi(A^r)\geq \omega(A,r),
\end{align}
in which
\begin{align}\label{b1b}
\omega(A,r)=\left\|\Phi(A)\right\|^r -\Bigg(\left\|\Phi(A)\right\|- \inf_{n\geq 1}\left\|\left(\Phi(A)+\frac{1}{n}-\Phi(A^r)^\frac{1}{r}\right)^{-1}\right\|^{-1}\Bigg)^r.
\end{align}
\end{lemma}
\begin{proof}
Not that, for $1/2\leq r <1$, the Choi--Davis inequality ensures that $\Phi(A)\geq \Phi(A^r)^\frac{1}{r}$. Hence $\Phi(A)+\frac{1}{n}> \Phi(A^r)^\frac{1}{r}$ for all positive integers $n$. We use an extension of the L\"{o}wner--Heinz inequality presented in \cite{MN}: If $A>B\geq 0$ and $r\in[0,1]$, then
 \begin{align}\label{ELH}
 A^r-B^r\geq \|A\|^r-\left(\|A\|-
 \left\|\left(A-B\right)^{-1}\right\|^{-1}\right)^r.
 \end{align}
Utilizing \eqref{ELH} with $\Phi(A)+\frac{1}{n}$ and $\Phi(A^r)^\frac{1}{r}$ instead of $A$ and $B$, respectively, we obtain
\begin{align*}
\left(\Phi(A)+\frac{1}{n}\right)^r-\Phi(A^r)&\geq \left\|\Phi(A)+\frac{1}{n}\right\|^r\\
&\quad -\left(\left\|\Phi(A)+\frac{1}{n}\right\|- \left\|\left(\Phi(A)+\frac{1}{n}-\Phi(A^r)^\frac{1}{r}\right)^{-1}\right\|^{-1}\right)^r.
\end{align*}
Taking the limits as $n\to\infty$, we get
\begin{align*}
\Phi(A)^r-\Phi(A^r)&\geq \left\|\Phi(A)\right\|^r\\
&\quad -\Bigg(\left\|\Phi(A)\right\|- \inf_{n\geq 1}\left\|\left(\Phi(A)+\frac{1}{n}-\Phi(A^r)^\frac{1}{r}\right)^{-1}\right\|^{-1}\Bigg)^r.
\end{align*}
Note that the sequence $\left\{\left(\Phi(A)+\frac{1}{n}-\Phi(A^r)^\frac{1}{r}\right)^{-1}\right\}_n$ is increasing, and hence the sequence $\left\{\left\|\left(\Phi(A)+\frac{1}{n}-\Phi(A^r)^\frac{1}{r}\right)^{-1}\right\|^{-1}\right\}$ is decreasing.
\end{proof}

To clarify Lemma \ref{MSM}, we give an example. Assume that the unital positive linear map $\Phi:\mathbb{M}_2\to\mathbb{M}_2$ is defined by $\Phi(A)= 1/2 \mathrm{Tr}(A) I_2$ and put
$$A=\left[\begin{array}{ccc}
 2 & 1 \\ 1 & 4
\end{array}\right].$$
Then, for $r=1/2$, the infimum in \eqref{b1b} is approximately equal to $0.18$ and $\omega(A,r)=\sqrt{3}-\sqrt{3-0.18}$.

%%%%%%%%%%%%%%%%%%%%%%%%%%%%%%%%%%%%%%%%%%%%%%%%%%%%%%%%%%%%%%%%%%%%%%%%%%%%%%%%
\begin{lemma}\label{asa}\cite{FNa}
 Let $A$ and $B$ be positive invertible operators such that $A-B\geq m>0$. Then
 \begin{align*}
A^\frac{p+r}{q}-\left(A^\frac{r}{2}B^pA^\frac{r}{2}\right)^\frac{1}{q}\geq \|A\|^\frac{p+r}{q}-\left\|A^{1+r}-m\|A^{-1}\|^{-r}\right\|^\frac{p+r}{q(1+r)}
 \end{align*}
 holds for every $p,r\geq0$ and $q\geq1$ with $(1+r)q\geq p+r$.
\end{lemma}

%%%%%%%%%%%%%%%%%%%%%%%%%%%%%%%%%%%%%%%%%%%%%%%%%%%%%%%%%%%%%%%%%%%%%%%%%%%%%%%%%%
Our next result provides a refinement of the asymmetric Kadison inequality.
\begin{theorem}\label{main22}
Let $\Phi$ be a unital positive linear map. If $X$ is a positive invertible operator, then
 {\small\begin{align}\label{main2}
 &\Phi(X^\alpha)^{1+\frac{\beta}{\alpha}}\nonumber\\
 &\geq \left|\Phi(X^\beta)\Phi(X^\alpha)\right|+
 \left\|\Phi(X^\alpha)^{1+\frac{\beta}{\alpha}}\right\|-
 \left\|\Phi(X^\alpha)^{2+\frac{\beta}{\alpha}}-
 \omega\left(X^\alpha,\frac{\beta}{\alpha}\right)\left\|\Phi(X^\alpha)^{\frac{-\beta}{\alpha}}\right\|^{\frac{-2\alpha}{\beta}}
 \right\|^{\frac{\beta+\alpha}{\beta+2\alpha}}
 \end{align}}
 for all $\alpha,\beta\geq0$ with $\beta< \alpha\leq 2\beta$.
\end{theorem}
\begin{proof}
Assume that $\beta< \alpha\leq 2\beta$ so that $\frac{\beta}{\alpha}\in[1/2,1]$.
 Lemma \ref{MSM} then shows that the inequality
 \begin{align*}
 \Phi(X)^\frac{\beta}{\alpha}-\Phi(X^\frac{\beta}{\alpha})\geq \omega\left(X,\frac{\beta}{\alpha}\right)
 \end{align*}
 is valid for every positive invertible operator $X$. Substituting $X$ by $X^\alpha$, we reach
 \begin{align*}
 \Phi(X^\alpha)^\frac{\beta}{\alpha}-\Phi(X^\beta)\geq \omega\left(X^\alpha,\frac{\beta}{\alpha}\right).
 \end{align*}

Now assume that $p=q=2$ and $r=2\frac{\alpha}{\beta}$ so that $(1+r)q\geq p+r$. If
$A=\Phi(X^\alpha)^\frac{\beta}{\alpha}$ and $B=\Phi(X^\beta)$, the hypotheses of Lemma \ref{asa} are satisfied for $m= \omega(X^\alpha,\frac{\beta}{\alpha})$. Since
 $$A^\frac{p+r}{q}=\left(\Phi(X^\alpha)^\frac{\beta}{\alpha}\right)^{1+\frac{\alpha}{\beta}}=
 \Phi(X^\alpha)^{1+\frac{\beta}{\alpha}}$$
 and
 $$\left(A^\frac{r}{2}B^pA^\frac{r}{2}\right)^\frac{1}{q}=
 \left(\left(\Phi(X^\alpha)^\frac{\beta}{\alpha}\right)^\frac{\alpha}{\beta}\Phi(X^\beta)^2
\left(\Phi(X^\alpha)^\frac{\beta}{\alpha}\right)^\frac{\alpha}{\beta}\right)^{\frac{1}{2}}
=\left|\Phi(X^\beta)\Phi(X^\alpha)\right|,$$
applying Lemma \ref{asa}, we get the desired result \eqref{main2}.
 \end{proof}
 %%%%%%%%%%%%%%%%%%%%%%%%%%%%%%%%%%%%%%%%%%%%%%%%%%%%%%%%%%%%%%%%%%%%%%%%%
As a consequence, let $\gamma=\frac{\beta}{\alpha}$. Employing $X^\frac{1}{\alpha}$ instead of $X$, we get the following result.
\begin{corollary}\label{lc}
Let $\Phi$ be a unital positive linear map. For every positive invertible operator and every $\gamma\in[1/2,1]$, it holds that
\begin{align*}
 \Phi(X)^{1+\gamma}\geq \left|\Phi(X^\gamma)\Phi(X)\right|+ \left\|\Phi(X)^{1+\gamma}\right\|-\big\|
 \Phi(X)^{2+\gamma}-\omega(X,\gamma)\|\Phi(X)^{-\gamma}\|^\frac{-2}{\gamma}\big\|^{\frac{1+\gamma}{2+\gamma}}.
\end{align*}
\end{corollary}
%%%%%%%%%%%%%%%%%%%%%%%%%%%%%%%%%%%%%%%%%%%%%%%%%%%%%%%%%%%%%%%%%%%%

It is noted in \cite{MN} that inequality \eqref{ELH} is sharp in the sense that when $A$ and $B$ are positive scalars of the identity operator, and then \eqref{ELH} becomes equality. Following the proof of Lemma \ref{MSM}, we realize that if $\Phi(A)=aI$ and $\Phi(A^r)^{1/r}=bI$ with $a>0$ and $b>0$, then \eqref{MSM1} turns into equality. Therefore inequality \eqref{MSM1} is sharp.

%%%%%%%%%%%%%%%%%%%%%%%%%%%%%%%%%%%%%%%%%%%%%%%%%%%%%%%%%%%%%%%%%%%%%%%%%%%%%%%%%%%%%%%%%%%%%%

 \end{document}